\documentclass[10pt]{amsart}
\usepackage {amssymb,latexsym}
\usepackage [all]{xy}
\newtheorem{theorem}{Theorem}[section]
\newtheorem{lemma}[theorem]{Lemma}
\newtheorem {proposition}[theorem]{Proposition}
\newtheorem {corollary} [theorem] {Corollary}

\theoremstyle{definition}
\newtheorem{definition}[theorem]{Definition}
\newtheorem{example}[theorem]{Example}

\newtheorem{problem}[theorem]{Problem}
\newtheorem{remark}[theorem]{Remark}
\numberwithin {equation}{section}

\def\cs{{$C^{\ast}$}}
\def\ra{{\rightarrow}}
\def\CG{{\mathbb{C}G}}
\def\q{{\mathbb{Q}}}
\def\z{{\mathbb{Z}}}
\def\n{{\mathbb{N}}}
\def\c{{\mathbb{C}}}
\def\r{{\mathbb{R}}}
\def\csr{{C^{\ast}_r}}
\def\ill{{^\infty_L}}
\def\H{{\mathcal{H}}}
\def\gh{{(G,H)}}

\def\f{{\varphi}}

\def\an{{\lhd_a}}

\def\wrt{{with respect to }}
\def\ba{{\backslash}}

\begin{document}
\title{Property (RD) for Hecke pairs}

\author{Vahid Shirbisheh}
%\address{Department of Mathematics, The University of Golestan, Gorgan, Golestan, Iran}
\email{shirbisheh@gmail.com}


%\date{\today}

\begin{abstract}
As the first step towards developing noncommutative geometry over Hecke \cs-algebras, we study property (RD) (Rapid Decay) for Hecke pairs. When the subgroup $H$ in a Hecke pair $\gh$ is finite, we show that the Hecke pair $\gh$ has (RD) if and only if $G$ has (RD). This provides us with a family of examples of Hecke pairs with property (RD). We also adapt Paul Jolissant's works in \cite{j1, j2} to the setting of Hecke \cs-algebras and show that when a Hecke pair $(G,H)$ has property (RD), the algebra of rapidly decreasing functions on the set of double cosets is closed under holomorphic functional calculus of the associated (reduced) Hecke \cs-algebra. Hence they have the same $K_0$-groups.
\end{abstract}
\maketitle
\footnote{The final publication is available at springerlink.com.}
% Shirbisheh, Vahid. Property (RD) for Hecke pairs. Math. Phys. Anal. Geom. Vol. {\bf 15}, No. 2 (2012), 111--???.
%\tableofcontents
%%%%%%%%%%%%%%%%%%%%%%%%%%%%%%%%%%%%%%%%%%%%%%%%%%      Section        %%%%%%%%%%%%%%%%%%%%%%%%%%%%%%%%%%%%%%%%%%%%%%%%%%%%%%%%%

\section {Introduction}
\label{sec:intro} Let $H$ be a subgroup of an arbitrary group $G$. It is called {\it almost normal in $G$} and is denoted by $H\an G$ if for every $g\in G$, the double coset $HgH$ is a finite union of its left cosets. A pair $(G,H)$ as above is called a {\it Hecke pair}. Elementary examples of almost normal subgroups are normal
subgroups, finite subgroups and subgroups of finite index. Besides
these examples, Hecke pairs have been appeared for the first time in
the theory of modular forms by considering $G=GL(2,\q)^+:=\{g\in
GL(2,\q); \det(g)>0\}$ and $H=SL(2,\z)$. We refer the reader to
Proposition 1.4.1 of \cite{bump}, for a proof of the fact that $H\an
G$. Hecke \cs-algebras were used by Jean-Beno\^{\i}t Bost and
Alain Connes, in \cite{bc}, in order to construct a \cs-dynamical
system illustrating  the class field theory of the field $\q$ of
rational numbers. Their Hecke pair consists of the group
$P_{\q}^+=\left\{ \left(
\begin{array} {rr}1&b\\0&a
\end{array}\right); a,b\in \q \,\text{and}\, a>0\right\}$ and its
subgroup $P^+_{\z}=\left\{ \left( \begin{array} {rr}1&n\\0&1
\end{array}\right); n\in \z \right\}$.
All groups in the above examples are considered as discrete groups.
In the setting of locally compact topological groups, one easily
observes that every compact open subgroup of a locally compact group
is almost normal. In fact, Kroum Tzanev has shown that every Hecke
pair can be ``replaced'' with a Hecke pair of this type and the
associated enveloping Hecke \cs-algebra would not change, see
Theorem 4.2 of \cite{tzanev}.

In this paper, we restrict our attention to the case that $G$ is a
discrete group. The set of all double cosets of a Hecke pair $(G,H)$
is denoted by $G//H$ and $\H(G,H)$ denotes the vector space of
finite support complex functions on $G//H$. An arbitrary set of
representatives of right cosets (resp.  double costs) of $H$ in $G$ is
denoted by $<H\ba G>$ (resp. $<G//H>$).

\begin{definition} Let $\gh$ be a Hecke pair. The vector space $\H \gh$ equipped with the following
convolution like product
\[
f_1 * f_2 (g):=\sum_{\gamma \in <H\ba G>} f_1(\gamma) f_2(\gamma^{-1}g), \quad f_1, f_2
\in \H(G,H).
\]
and the involution
\[
f^*(g):=\overline{f(g^{-1})}
\]
is called the {\it Hecke algebra of $(G,H)$}.
\end{definition}
Let $\c(H\ba G)$ denote the vector space of finite support complex
functions on the set of right cosets of $H$ in $G$. By extending the
definition of the above convolution product, it is endowed with an
$\H\gh$-module structure as follows:
\begin{eqnarray*}
\H\gh\times \c(H\ba G) &\ra& \c(H\ba G)\\
f_1 * f_2 (g)&:=&\sum_{\gamma\in <H\ba G>} f_1(\gamma) f_2(\gamma^{-1}g),
\end{eqnarray*}
for all $g\in H\ba G$. This gives rise to a $\ast$-representation $\lambda: \H(G,H) \ra
\mathcal{B}(\ell^2(H\ba G))$ defined by left convolutions;
\begin{equation}
\label{f:regrep} \lambda(f)(\xi)(g):=(f*\xi)(g)=\sum_{\gamma\in <H\ba G>}
f(\gamma) \xi(\gamma^{-1}g),
\end{equation}
for all $f \in \H(G,H)$, $\xi \in \ell^2(H\ba G)$ and $g\in H\ba G$. Due to the fact that each double coset is the union of only finitely many right cosets, one easily checks that every function in $\H(G,H)$ is mapped to a bounded operator by $\lambda$. We call this map the {\it regular representation of the Hecke pair $(G,H)$}. For $f\in \H(G,H)$, the norm of $\lambda(f)$ in $\mathcal{B}(\ell^2(H\ba G))$ is called the {\it convolution norm of $f$} and is denoted by $\|\lambda(f) \|$.
\begin{definition}
The norm closure of the image of the regular representation of a
Hecke pair $(G,H)$ is called the {\it reduced Hecke \cs-algebra of
$(G,H)$} (or shortly Hecke \cs-algebra of $(G,H)$) and is denoted by
$\csr\gh$.
\end{definition}
We refer the reader to Rachel Hall's thesis, \cite{Hall}, for the
definition of full Hecke \cs-algebras and a study of the problem of
the existence of full Hecke \cs-algebras. We also refer the reader
to \cite{tzanev} for the definition of the enveloping \cs-algebra of
a Hecke pair. The Hecke algebra $\H\gh$ also acts on $\ell^2(G/H)$ by right convolution, see \cite{Hall} for details.

It is clear that when $H$ is normal in $G$, the above definitions
coincide correspondingly with the definitions of the group algebra,
the convolution product and norm, the regular representation, and
the reduced group \cs-algebra of the quotient group $G/H$. This
point of view motivates our program to generalize the concepts and
tools of noncommutative geometry over group \cs-algebras to the more
general setting of Hecke \cs-algebras. We note that Kroum Tzanev
started a similar program to reformulate the Baum-Connes conjecture in
the setting of Hecke \cs-algebras in his thesis, see
\cite{tzanevthesis}. In the following, we explain why we begin with
the study of property (RD) for Hecke pairs.

Many notions of noncommutative geometry over \cs-algebras are
defined on a specific type of dense subalgebras of \cs-algebras,
often called {\it smooth subalgebras of \cs-algebras}. The name
comes from the commutative case that the algebra $C^\infty(M)$ of
smooth complex functions on a compact smooth manifold $M$ is dense
in the \cs-algebra $C(M)$ of continuous complex functions on $M$.
The main feature of these smooth subalgebras is that they are closed
under holomorphic functional calculus of the containing
\cs-algebras.

\begin{definition} An involutive dense subalgebra $\mathcal{A}$ of a
\cs-algebra $A$ is called {\it smooth} if for every element $a\in
\mathcal{A}$ and every holomorphic function $f$ defined over an open
set containing the spectrum of $a$ in $A$, $f(a)$ belongs to
$\mathcal{A}$. In this case, we also say that $\mathcal{A}$ is {\it
closed under holomorphic functional calculus of $A$}.
\end{definition}
\begin{remark}
If $\mathcal{A}$ is a smooth subalgebra of a \cs-algebra $A$, then
it has the {\it property of spectral permanence}, see \cite{ar}, namely
the spectrum of each element of $\mathcal{A}$ in $\mathcal{A}$ is
the same as its spectrum in $A$, see proposition 3.1.3 in \cite{bl}.
The converse is true if $\mathcal{A}$ is endowed with a topology for
which the Cauchy integral of holomorphic functional calculus
converges. For instance, if $\mathcal{A}$ is an involutive dense
Banach algebra and has the property of spectral permanence in $A$,
then it is a smooth subalgebra of $A$. This was noted by
Jean-Beno\^{\i}t Bost in the discussion after Theorem 1.3.1 of
\cite{b}.
\end{remark}
If $M_n(\mathcal{A})$ is a smooth subalgebra of $M_n(A)$ for every
positive integer $n$, then $\mathcal{A}$ is called a {\it local
\cs-algebra}. This property ensures that $\mathcal{A}$ is similar
enough to the \cs-algebra $A$ to carry many features of $A$. For
example, both have the same $K_0$-groups, see \cite{bl}.

\begin{remark}
In the appendix of \cite{b}, J.-B. Bost introduced an interesting method to show under some minor conditions spectral permanence implies the equality of $K$-groups too. Clearly, his method is more general than using stability under holomorphic functional calculus, because it works in the setting of Fr\'{e}chet algebras. Reader can find more details in Theorem A.2.1 of \cite{b} about his approach.
\end{remark}

The main reason for constructing smooth subalgebras inside \cs-algebras is that working with smooth subalgebras instead of \cs-algebras has the advantage that they are more capable of algebraic constructions like Connes' cyclic cohomology and geometric constructions like Connes' spectral triple, \cite{c85, c94}. Therefore, in many situations, the first step to study the noncommutative geometry of a \cs-algebra is to define a smooth
subalgebra.

There are many ways to define smooth subalgebras of \cs-algebras. For example, the subalgebra of smooth elements with respect to an action of a Lie group on a \cs-algebra is smooth. For more general constructions of smooth subalgebras using differential seminorms and derivations, see \cite{blcu}. Whenever the \cs-algebra under consideration is related to a group, for example group \cs-algebras and crossed product \cs-algebras, harmonic analysis provides us with a method to define smooth subalgebras. From Fourier analysis one knows that the algebra of rapidly decreasing (also called Schwartz)
functions on $\mathbb{Z}$ is isomorphic to the algebra $C^\infty(\mathbb{T})$ of smooth functions on the unit circle $\mathbb{T}\subset\c$. Since $C^\infty(\mathbb{T})$ is a smooth subalgebra of $C(\mathbb{T})$ and this \cs-algebra is isomorphic to $C^\ast(\z)$, one can consider the algebra of rapidly decreasing functions on $\z$ as a smooth subalgebra of $C^\ast(\z)$. The idea of considering the subalgebra of rapidly decreasing functions on a group as the smooth subalgebra of its reduced group \cs-algebra was generalized by Paul Jolissaint for groups possessing property (RD) in \cite{j1, j2}.

The main purpose of this article is to show, in details, that this idea works similarly for Hecke pairs and Hecke \cs-algebras. Therefore, our main result is that when a Hecke pair $\gh$ possesses property (RD) with respect to a length function $L$, the subalgebra of rapidly decreasing functions on the set of double cosets with respect to $L$ is a smooth subalgebra of $\csr\gh$, see Proposition \ref{pro:HT} and results after that. As the first application of this fact, we show that these algebras have the same $K_0$-groups, see Corollary \ref{cor:Kgroups}. These are done in Section \ref{sec:rdfunctions}. Our proofs are a modification of the original proofs, given by Paul Jolissaint, for the general framework of Hecke \cs-algebras.

Section \ref{sec:RD} is devoted to the definition of property (RD) and equivalent definitions. Some remarks and complementary discussions are also given in Section \ref{sec:RD}. In order to give some examples of non-trivial Hecke pairs with property (RD), we show, in Theorem \ref{thm:fsub}, that if $H$ is a finite subgroup of a group $G$, then the Hecke pair $\gh$ has (RD) if and only if the group $G$ has (RD).
%%%%%%%%%%%%%%%%%%%%%%%%%%%%%%%%%%%%      SECTION      property (RD)                %%%%%%%%%%%%%%%%%%%%%%%%%%%%%%%%%%%%%%%%
\section{Property (RD)}
\label{sec:RD} The definition of property (RD) for a group is based on
the notion of a length function. For a discrete group $G$, a {\it
length function on $G$} is a function $L: G \ra [0,\infty[$ which
satisfies following conditions for all $g,h \in G$:
\begin{itemize}
\item [(i)] $L(gh)\leq L(g)+L(h)$,
\item [(ii)] $L(g)= L(g^{-1})$,
\item [(iii)] $L(1)=0$.
\end{itemize}
A length function $L$ on $G$ gives rise to a {\it weighted
$\ell^2$-norm}, for every $s>0$, as follows:
\[
\parallel f \parallel_{s,L}:= \left( \sum_{g\in G} \mid f(g)\mid^2
(1+L(g))^{2s}\right)^{\frac{1}{2}}, \qquad \forall f\in\c G.
\]
Let $\lambda :\CG \ra \mathcal{B}(\ell^2(G))$ be the left regular
representation of $G$. The {\it convolution norm} of a function $f$
in $\CG$ is defined by $\parallel \lambda(f) \parallel$.
\begin{definition} \label{def:RD}  We say a group $G$ has {\it property (RD)}
if there exist a length function $L$ on $G$ and  positive real
numbers $C$ and $s$ such that the {\it Haagerup inequality};
\begin{equation}
\label{eq:Hineq1}
\parallel \lambda(f) \parallel\leq C\parallel f \parallel_{s,L}
\end{equation}
holds for all $f\in \c G$.
\end{definition}
Uffe Haagerup introduced and proved Inequality \ref{eq:Hineq1}, for $C=s=2$, for free groups of rank $n\geq 2$ equipped with the word length function in \cite{haag1}. Afterwards, Jolissaint gave the formal definition of property (RD) and proved several statements related to this property and discussed several examples of groups possessing this property in \cite{j2}. He also showed that in the presence of property (RD) the subalgebra of rapidly decreasing functions is smooth in the reduced group \cs-algebra, \cite{j1}. Pierre de la Harpe proved property (RD) for hyperbolic groups in \cite{harpe}. As it was noted in \cite{cm, j1}, when a discrete group $G$ has property (RD) a cyclic cocycle of $\c G$, (an element of the cyclic cohomology group of $\c G$), satisfying some  additional conditions extends to an $n$-trace over $\csr(G)$. Thus, it defines an index map $K_0(\csr(G))\ra \c$, (in the terminology of noncommutative geometry and cyclic cohomology). This phenomenon was used by Alain Connes and Henri Moscovici to prove the Novikov conjecture for hyperbolic groups in \cite{cm}. Property (RD) also appears in Vincent Lafforgue's proof of the Baum-Connes conjecture for a new family of groups, see \cite{l1} for more details. More recently, property (RD) has been applied to define and study noncommutative metrics over the state spaces of several reduced group \cs-algebras by Cristina Antonescu and Erik Christensen in \cite{ach}. This latter application of property (RD) is another motivation of the present work and will be discussed in our subsequent paper. We also note that property (RD) for discrete quantum groups was studied by Roland Vergnioux in \cite{ve}.

\begin{remark}
\label{rem:amen} Jolissaint has shown that when $G$ is an amenable group, $G$ has (RD) if and only if $G$ is of polynomial growth, see Corollary 3.1.8 of \cite{j2}. This is the only obstruction known for property (RD) yet. In Proposition 6 of \cite{valette}, Alain Valette refined this result to investigate a family of groups which are not amenable.
\end{remark}

Now, we extend the definition of property (RD) to Hecke pairs, and later on we will use this notion to define smooth subalgebras in Hecke \cs-algebras.

\begin{remark}
\label{rem:NL} Given a length function $L$ on $G$, It is easy to check that $N_{L}:=\{g\in G; L(g)=0\}$ is a subgroup of $G$. Let $H$ be a subgroup of $N_{L}$. Then, for every $h_1, h_2 \in H$ and $g\in G$, we have $L(g)=L(h_1^{-1}h_1g)\leq L(h_1^{-1})+L(h_1g)=L(h_1g)\leq L(h_1) +L(g)=L(g)$, and consequently $L(g)=L(h_1g)$. Similarly, one can show that $L(g)=L(h_1 g h_2)$.
\end{remark}
This observation leads us to the following definition:
\begin{definition} A {\it length function on a Hecke pair (G,H)} is a
length function $L$ on $G$ such that $H \leq N_{L}$.
\end{definition}
Similar to group algebras, for every $s>0$, we define the {\it
weighted $\ell^2$-norm associated with $s$ and $L$} on finite support functions on
the set $G//H$ of double cosets as follows
\[
\parallel f \parallel_{s,L}:= \left( \sum_{g\in <H\ba G>} \mid f(g)\mid^2
(1+L(g))^{2s}\right)^{\frac{1}{2}}, \qquad \forall f:G//H\ra \c.
\]
\begin{definition}
\label{def:HRD} We say a Hecke pair $(G,H)$ has {\it property (RD)} if
there exist a length function $L$ on $(G,H)$ and positive real
numbers $C$ and $s$ such that the Haagerup inequality;
\begin{equation}
\label{eq:Hineq2}
\parallel \lambda(f) \parallel\leq C\| f \|_{s,L}
\end{equation}
holds for all $f\in \H(G,H)$.
\end{definition}

\begin{remark}
\label{r:altdef}
For $f\in \H\gh$, we defined property (RD) based on the norm $\|f\|_{s,L}$ which uses a sum over right cosets. Alternatively, one can define property (RD$^\prime$) using the norm $\parallel f \parallel_{s,L}^{\prime}:= \left( \sum_{g\in <G//H>} \mid f(g)\mid^2 (1+L(g))^{2s}\right)^{\frac{1}{2}}$. Clearly, we have $\|-\|^\prime_{s,L}\leq \|-\|_{s,L}$, so property (RD$^\prime$) \wrt a length function $L$ implies property (RD) \wrt $L$.
\end{remark}

\begin{proposition}
If a Hecke pair $\gh$ has property (RD$^\prime$) \wrt a length function $L$ then there are positive constants $D$ and $t$ such that $| HgH/Hg |\leq D(1+L(g))^{t}$ for all $g\in G$.
\end{proposition}
First, we note that, in the presence of the above inequality, property (RD) clearly implies property (RD$^\prime$).
\begin{proof}
Let $\gh$ has property (RD$^\prime$) \wrt $L$ and with some positive constants $C$ and $s$. For $g\in G$, let $\delta_g \in \H\gh$ denote the characteristic function of the double coset $HgH$ and let $\delta_1 \in \c(H\ba G)$ denote the characteristic function of the right coset $H$. Then we have $| HgH/Hg |=\|\delta_g\|_2^2=\|\delta_g\ast \delta_1\|_2^2\leq \|\lambda(\delta_g)\|^2\leq C^2\|\delta_g\|_{s,L}^2 = C^2(1+L(g))^{2s}$, where the norm $\|-\|_2$ is the norm of $\ell^2(H\ba G)$.
\end{proof}

We do not try to verify all statements for both property (RD) and property (RD$^\prime$). But one can easily check that the content of Section \ref{sec:rdfunctions} holds similarly for property (RD$^\prime$). To conform the content of the rest of this section with property (RD$^\prime$), one should modify Proposition \ref{pro:def} according to this property.

\begin{example}
\begin{itemize}
\item[(i)] When $H$ is normal in $G$, one notes that property (RD) for a Hecke pair $(G,H)$ coincides with property (RD) of $G/H$, in particular when $H=\{e\}$.
\item[(ii)] Let $H$ be a subgroup of $G$ of finite index, say $n$. Then, using Proposition 6.12 of \cite{fo}, we have $\|-\|_1^2\leq n \|-\|_2^2 = n \|-\|_{s,L}^{2}$, where $L$ is the constant zero length function, $s$ is any positive real number and the norms $\|-\|_1$ and $\|-\|_2$ are $\ell^1$ and $\ell^2$ norms of $H\ba G$, respectively. Given $f\in \H\gh$, $f$ is an element of $\c(H\ba G)$ too and by repeating the proof of Young's inequality in our setting, see Proposition 8.7 of \cite{fo}, we have $\|\lambda(f)\|\leq \|f\|_1$. Therefore, in this case, the Haagerup inequality \ref{eq:Hineq2} holds for $L=0$, $C=\sqrt{n}$ and any positive real number $s$.
\end{itemize}
\end{example}

One possible application of property (RD) for Hecke pairs is the
following problem which is motivated by Remark \ref{rem:amen}.

\begin{problem}
Define the notion of the growth of a Hecke pair, compatible with the definition of growth for groups. Find a relationship between the amenability of Hecke pairs as defined by Tzanev in \cite{tzanev} and property (RD) of Hecke pairs. Look for examples of non-amenable Hecke pairs using this relationship.
\end{problem}

There are several conditions equivalent to property (RD) which can be generalized to the setting of Hecke pairs and are necessary for our discussion. First, we need some notations. The subsets of non-negative real functions in $\H\gh$, $\c(H\ba G)$, and $\c G$ are denoted by $\H_+\gh$, $\r_+(H\ba G)$, and $\r_+ (G)$, respectively. For a length function $L$ on a Hecke pair $\gh$ and for every non-negative real number $r$, we define
\[
B_{r,L}\gh:= \{ HgH\in <G//H>; L(g)\leq r\},
\]
\[
C_{r,L}\gh:= \{ HgH \in <G//H>; r\leq L(g)< r+1\}.
\]
We also denote the similar sets in $<H\ba G>$ and $G$ by $B_{r,L}(H\ba G)$, $C_{r,L}(H\ba G)$ and  $B_{r,L}(G)$, $C_{r,L}(G)$, respectively. For $f\in \c(H\ba G)$ or $f\in \H\gh$, the norm of $f$ in $\ell^2(H\ba G)$ is denoted by $\|f\|_2$.

\begin{proposition}\label{pro:def}
Let $L$ be a length function on a Hecke pair $\gh$. Then the following statements are equivalent:
\begin{itemize}
\item [(i)] The Hecke pair $\gh$ has property (RD) \wrt $L$.
\item [(ii)] There exists a polynomial $P$ such that for any $r>0$
and $f\in \H_+\gh$ so that $\text{supp} f\subseteq B_{r,L}\gh$, we have
\[
\|\lambda(f)\|\leq P(r) \| f\|_{2}.
\]
\item [(iii)] There exists a polynomial $P$ such that for any $r>0$,
$f\in \H_+\gh$ so that $\text{supp}f\subseteq B_{r,L}\gh$, and $k\in
\r_+(H\ba G)$, we have
\[
\|f\ast k\|_2\leq P(r)\| f\|_{2} \|k\|_{2}.
\]
\end{itemize}
\end{proposition}
\begin{proof}
Let $s,C$ be positive real numbers such that $\|\lambda(-)\|\leq
C\|-\|_{s,L}$ over $\H\gh$. For $r\geq0$, assume $f\in \H_+\gh$ is
supported in $B_{r,L}\gh$. Then we have
$$\|\lambda(f)\|\leq
C\|f\|_{s,L}=C\left( \sum_{g\in B_{r,L}(H\ba G)} |f(g)|^2 (1+L(g))^{2s}
\right)^{1/2}$$
$$\leq C\left( \sum_{g\in B_{r,L}(H\ba G)} |f(g)|^2 (1+r))^{2s}
\right)^{1/2}=C (1+r)^{s}\| f\|_2.$$
Hence (ii) follows from (i) by
considering $P(r)=C(1+r)^s$.

Conversely, Let (ii) hold for some polynomial $P$. One can easily find two positive numbers $C,s$ such that $P(n)\leq Cn^{s-1}$ for all $n\in \n$. Let $r$ be a non-negative real number and let $f\in\H\gh$ so that $\text{supp}f\subseteq B_{r,L}\gh$. It is easy to check that $\|\lambda(f)\|\leq \|\lambda(|f|)\|$. Combining this inequality with (ii) and the fact that $\|\, |f|\, \|_2= \| f\|_2$, we get $\|\lambda(f)\|\leq P(r)\| f\|_2$. Now, let $f\in \H\gh$ be arbitrary and let $\chi_n$ denote the characteristic function of $C_{n-1,L}\gh$ for all $n\in \n$. Then, $\|\lambda(f)\| = \|
\sum_{n=1}^\infty \lambda(f\chi_n)\| \leq \sum_{n=1}^\infty \|\lambda(f\chi_n)\| \leq \sum_{n=1}^\infty P(n)\| f\chi_n\|_2\leq \sum_{n=1}^\infty Cn^{s-1}\| f\chi_n\|_2=C \sum_{n=1}^\infty n^{-1} n^s \| f\chi_n\| _2$. By applying Cauchy-Schwartz inequality of $\ell^2(\n)$, we obtain $\|\lambda(f)\|\leq C \left( \sum_{n=1}^\infty n^{-2}\right)^{1/2} \left(\sum_{n=1}^\infty n^{2s}\| f\chi_n \|_2^2 \right)^{1/2}$. Setting $C^\prime := C \left( \sum_{n=1}^\infty n^{-2}\right)^{1/2}$, we have $\|\lambda(f)\|\leq C^\prime \left(\sum_{n=1}^\infty n^{2s}\| f\chi_n\| _2^2 \right)^{1/2} \leq C^\prime \left(\sum_{n=1}^\infty (1+L(g))^{2s}\sum_{g\in C_{n-1,L}(H\ba G)} |f(g)|^2 \right)^{1/2}$, where the last inequality follows from the facts that $n\leq L(g) +1$ for every $g\in C_{r-1,L}(H\ba G)$ and $\| f\chi_n\|_2^2=\sum_{g\in C_{n-1,L}(H\ba G)} |f(g)|^2$. Finally, Since $\bigcup_{n=1}^\infty C_{n-1, L}(H\ba G)=H\ba G$, we obtain $\|\lambda(f)\|\leq C^\prime \left( \sum_{g\in <H\ba G>}(1+L(g))^{2s} |f(g)|^2 \right)^{1/2}= C^\prime \|f\|_{s,L}$. The equivalence between (ii) and (iii) is easy and is left to the reader.
\end{proof}

The above proposition and its proof, in the setting of groups, appeared implicitly in \cite{haag1}, for instance, see Lemmas 1.4 and 1.5 in there. Condition (ii) in the above also appeared as the definition of property (RD) in \cite{l2}. The proof of this proposition is an adaptation of the proof for the case of groups which was taken from \cite{chat}, see also proposition 1.2 of \cite{chatru}.

For our discussion in the rest of this section, we need to recall some more definitions from \cite{j2} and set some notations. Let $L_1$ and $L_2$ be two length functions on some group $G$. We say $L_1$ {\it dominates} $L_2$ if there exist positive real numbers $a,b$ such that $L_2\leq a L_1+b$. In this case $1+L_2\leq (b+1) + a L_1$ and setting $M:=\max\{b+1, a\}$, we get $1+L_2\leq M(1+ L_1)$. This implies that $\|-\|_{s,L_2}\leq M^{s}\|-\|_{s,L_1}$
for all $s>0$. Thus if $G$ has property (RD) \wrt $L_2$ then  it has it \wrt $L_1$ as well. For instance, the word length function of a finitely generated group $G$ dominates other length functions of $G$, see Lemma 1.1.4 of \cite{j2}. So, if $G$ does not have (RD) \wrt the word length function it does not have (RD) \wrt other length functions either. If two length functions $L_1$ and $L_2$ dominate each other, we say $L_1$ and $L_2$ are {\it equivalent}. It is clear that the above discussion holds for Hecke pairs too. In particular, if a Hecke pair has property (RD) with respect to a length function $L$ it has property (RD) with respect to any length function equivalent to $L$.

%%%%%%%%%%%%%%%%%%%%%%%%%%%%%%%%%%%%%%%%%%%%%%%%%%  Section  commensurable subgroups %%%%%%%%%%%%%%%%%%%%%%%%%%%%%%%%%%%%%%%%%%%%%%%%%%%%%%%%%%%%%
The following theorem relates the property (RD) of groups and Hecke pairs consisting of groups and their finite subgroups.

\begin{theorem}
\label{thm:fsub}
Let $H$ be a finite subgroup of a group $G$. Then $G$ has property (RD) if and only if the Hecke pair $\gh$ has property (RD).
\end{theorem}

\begin{proof} Let $n$ be the order of $H$. Suppose $G$ has (RD) \wrt a length function $L$. Using Lemma 2.1.3 of \cite{j2}, we can replace $L$ with an equivalent length function which is zero on $H$.  Let $P$ be the polynomial appeared in Part (iii) of Proposition \ref{pro:def}. Let $f\in \H_+\gh$ with $\text{supp} f\subseteq B_{r,L}\gh$ and let $k\in\r_+(H\ba G)$. We define $\tilde{k}\in \r_+(G)$ (resp. $\tilde{f} \in \r_+(G)$) by $\tilde{k}(x):=k(Hx)$  (resp. $\tilde{f}(x):=f(HxH)$) for all $x\in G$. Then, we have
\[
\|\tilde{k}\|_2^2=n \|k\|_2^2,
\]
where the norms are taken in $\ell^2(G)$ and $\ell^2(H\ba G)$, respectively. Also, we note that $\tilde{f}\in B_{r,L}(G)$ and we have
\[
\|\tilde{f}\|_2^2= n \| f\|_2^2,
\]
where the norms are taken in $\ell^2(G)$ and $\ell^2(H\ba G)$, respectively. Now, we have
\begin{eqnarray*}
\|f\ast k\|_2^2 &=& \sum_{y\in <H\ba G>}\left(\sum_{x\in<H\ba G>} f(x)k(x^{-1}y)\right)^2\\
&=& \sum_{y\in <H\ba G>}\left(\frac{1}{n}\sum_{x\in G} \tilde{f}(x)\tilde{k}(x^{-1}y)\right)^2\\
&=& \frac{1}{n^3}\sum_{y\in G}\left(\tilde{f}\ast\tilde{k}(y)\right)^2\\
&=& \frac{1}{n^3}\|\tilde{f}\ast\tilde{k}\|_2^2\\
&\leq& \frac{1}{n^3}P(r)^2\|\tilde{f}\|_2^2\|\tilde{k}\|_2^2\\
&=& \frac{1}{n} P(r)^2\| f\|_2^2 \|k\|_2^2.
\end{eqnarray*}
Thus $\gh$ has (RD) \wrt $L$.

Conversely, let $\gh$ has (RD) \wrt $L$ and let $P$ be the polynomial in Part (iii) of Proposition \ref{pro:def}. Let $H=\{ h_1, \cdots, h_n\}$. For $f\in \r_+(G)$ with $\text{supp}f \subseteq B_{r,L}(G)$, define $\bar{f}\in \H_+\gh$ by $\bar{f}(HgH):=\sum_{i,j=1}^n f(h_i g h_j)$. For $m\in\n$, let $c(m)$ be the least constant for which $ \left(\sum_{i=1}^m x_i\right)^2\leq c(m)
\sum_{i=1}^m x_i^2$ for all $x_i\geq 0$. One computes
\begin{eqnarray*}
\| \bar{f}\|_2^2 &=&\sum_{g\in <H\ba G>} \left( \sum_{i,j=1}^n f(h_i g h_j)  \right)^2\\
&\leq& c(n^2) \sum_{g\in <H\ba G>} \sum_{i,j=1}^n \left( f(h_i g h_j)  \right)^2\\
&\leq& n c(n^2)\| f\|_2^2.
\end{eqnarray*}
For $k\in \r_+ (G)$, define $\bar{k}\in\r_+(H\ba G)$ by $\bar{k} (Hg):= \sum_{i=1}^n k(h_i g)$. A similar computation as above shows that $\|\bar{k}\|_2^2 \leq  c(n)\| k\|_2^2$. We also note that $f\leq \tilde{\bar{f}}$ and $k\leq \tilde{\bar{k}}$.
Therefore, for every $g\in G$, we have $f\ast k (g)=\sum_{x\in G}f(x)k(x^{-1}
g) \leq \sum_{x\in <H\ba G>}\tilde{\bar{f}}(x)\tilde{\bar{k}}(x^{-1}
g)=\tilde{\bar{f}} \ast \tilde{\bar{k}}(g)$. Hence
\begin{eqnarray*}
\|f\ast k\|_2^2&\leq & \| \tilde{\bar{f}}\ast \tilde{\bar{k}}\|_2^2\\
&= & n^{2} \|\bar{f}\ast \bar{k}\|_2^2\\
&\leq & n^{2} P(r)^2 \| \bar{f}\|_2^2\| \bar{k}\|_2^2\\
&\leq & n^3 c(n)c(n^2) P(r)^2 \|f\|_2^2\| k\|_2^2.
\end{eqnarray*}
This shows that $G$ has (RD) \wrt $L$.
\end{proof}

\begin{remark}
\label{r:infsub}
When $H$ is an infinite almost normal subgroup of $G$, the Hecke pair $\gh$ and the group $G$ cannot have property (RD) \wrt a common length function $L$. The reason is that if $L$ is a length function on the Hecke pair $\gh$, then $L$ is zero on $H$ and consequently $H$ cannot have (RD) \wrt $L$, see Example 1.6 of \cite{chat}. Therefore, $G$ cannot have (RD) \wrt $L$, see  Proposition 2.1.1 of \cite{j2}.
\end{remark}

The above theorem generalizes Proposition 2.1.4 of \cite{j2}. A generalization of the above theorem will be given in \cite{s1}. Every group $G$ possessing property (RD) with some finite subgroup $H$ gives rise to an example for Hecke pairs with property (RD). Constructions in Section 2 of \cite{j2} provide us with a number of examples of this kind. Here, we content ourself with an example coming from semidirect products.

\begin{definition}
Let $E$ be a group generated by a finite set $S=S^{-1}$ and let $l$ be the word length function on $E$ \wrt $S$. For $\alpha\in \text{Aut}(E)$, the {\it amplitude of $\alpha$ \wrt $S$} is the number $a(\alpha):=\max_{s\in S}l(\alpha(s))$. Let $F$ be another finitely generated group with a word length function $L$. We say a map $\theta :F\ra \text{Aut}(E)$ has {\it polynomial amplitude} if there exist positive constants $c$ and $r$ such that
$a(\theta(f))\leq c(1+L(f))^r$ for all $f\in F$.
\end{definition}

\begin{proposition} [\cite{j2}, Corollary 2.1.10] Let $F$ and $E$ be two finitely generated groups and let $\alpha$ be an action of $F$
on $E$ of polynomial amplitude. If $E$ and $F$ have property (RD),
then so does the semidirect product $E\rtimes_\alpha F$.
\end{proposition}
Clearly, any action of a finite group on a finitely generated group
is of polynomial amplitude. Thus we have the following corollary:

\begin{corollary}
Let $F$ be a finite group acting on a finitely generated group $G$.
If $G$ has property (RD), then the Hecke pair $(G\rtimes F, F)$ has
property (RD).
\end{corollary}

\begin{example}
Consider the infinite dihedral group $D_\infty=\z\rtimes \z/2\z$.
The Hecke pair $(D_\infty, \z/2\z)$ has property (RD).
\end{example}

We note that the Hecke pair $(D_\infty, \z/2\z)$ is actually a {\it Gelfand pair}, namely, its Hecke algebra is commutative.

\begin{remark}
Theorem \ref{thm:fsub} can be applied to give non-examples too. One
only needs to find a group $G$ not having (RD) with some non-normal
finite subgroup. To find such a group, we use Corollaries 3.1.8 and
3.1.9 of \cite{j2}. For example, $SL_3(\z)$ does not have (RD). To
create a situation as above, set $T:=\left(\begin{array}{ccc} 0 & 1
& 0\\1&0&0\\0&0&-1
\end{array}\right)$ and $S:=\left(\begin{array}{ccc} 0 & 1 & 0\\-1&0&0\\0&0&1
\end{array}\right)$ and let $H$ denote the subgroup generated by $T$.
Since $STS^{-1}\notin H$, $H$ is not normal in $SL_3(\z)$. Thus, the
Hecke pair $(SL_3(\z),H)$ is a non-trivial Hecke pair which does not
have (RD). If $G$ does not have a non-normal finite subgroup or it is
difficult to find such a subgroup, one can try the semidirect
product of $G$ with a finite group $H$ which $H$ acts non-trivially
on $G$. If there is such a finite group $H$, then the Hecke pair
$(G\rtimes H, H)$ does not have (RD).
\end{remark}

%%%%%%%%%%%%%%%%%%%%%%%%%%%%%%%%%%%%%      Section  (RD) functions       %%%%%%%%%%%%%%%%%%%%%%%%%%%%%%%%%%%%%%%%

\section{The smooth subalgebra of rapidly decreasing functions}
\label{sec:rdfunctions}

In this section, $\gh$ is a Hecke pair equipped with a length
function $L$. We extend basic definitions of \cite{j1, j2} to this
setting. For $s\in\r$, the {\it Sobolev space of order $s$ with
respect to $L$} is defined as
\[
H^s_L\gh:=\{f:G//H\ra \c ; \|f\|_{s,L}< \infty \}.
\]
\begin{remark}
\label{rem:rdf1}
\begin{itemize}
\item [(i)] The space $H^s_L\gh$ can be considered as a Hilbert
space equipped with the the inner product defined by
\[
\ \qquad  \langle f_1,f_2\rangle_{s,L}:= \sum_{g\in <H\ba G>} f_1(g)\overline{f_2(g)} (1+L(g))^{2s}, \qquad \forall f_1, f_2\in H^s_L\gh.
\]
The completeness of $H^s_L\gh$ follows from a similar argument as
the proof of completeness of $L^2$ spaces because the above inner
product has been defined by an integral (with respect to the
counting measure).
\item[(ii)] One easily observes that $\|-\|_{s,L}\leq \| -\|_{t,L}$ for
$s\leq t$. Therefore, the spaces $H^s_L\gh$ are decreasing with
respect to the parameter $s$.
\item[(iii)] Our discussion before Theorem \ref{thm:fsub} shows that
if $L_1$ and $L_2$ are two length functions on a Hecke pair $\gh$
and $L_1$ dominates $L_2$, then $H^s_{L_1}\gh \subseteq  H^s_{L_2}\gh$ for
all $s\geq 0$.
\end{itemize}
\end{remark}
\begin{definition}
The {\it space of rapidly decreasing functions associated with the
Hecke pair $\gh$} with respect to $L$ is
\[
H^\infty_L\gh:=\bigcap_{s\geq 0} H^s_L\gh.
\]
\end{definition}
The norms $\{\|-\|_{s,L}\}_{s\geq0}$ induce a locally convex
topology on $H^\infty_L\gh$ and by Remark \ref{rem:rdf1}, it is the
same topology as when the parameter $s$ runs through natural numbers.
Therefore, $H^\infty_L\gh$ is a countably normed space and is
complete by Proposition 2.4 of \cite{bec1}. Thereby, $H^\infty_L\gh$
becomes a Fr\'echet space.
\begin{remark}
If a Hecke pair $\gh$ has property (RD) with respect to a length
function $L$, then $H\ill\gh\subseteq \csr\gh$. This follows from the
fact that $\H\gh$ generates $H^s_L\gh$ as a Hilbert space for all
$s\geq 0$ and for $s$ sufficiently large we have $H^s_L\gh\subseteq
\csr\gh$ because of the Haagerup inequality \ref{eq:Hineq2}.
\end{remark}

The aim of this section is to modify the contents of \cite{j1}
according to the setting of Hecke pairs to show that if a Hecke pair
$\gh$ possesses property (RD) with respect to a length function $L$,
then $H^\infty_L\gh$ is a smooth subalgebra of the Hecke \cs-algebra
$C^*_r\gh$.

For $r\geq 0$, let $P_r$ be the orthogonal projection on the closed
span of the set $\{\delta_g; g\in B_{r,L}(H\ba G) \}$ in
$\ell^2(H\ba G)$, where $\delta_g$ is the characteristic function of
the right coset $Hg$ in $H\ba G$. For $0<\alpha<1$ and $q, N\in
\mathbb{N}$, we define a map $\rho_{\alpha, q, N}: \csr\gh\ra
[0,\infty[$ by
\[
\rho_{\alpha, q, N}(a):=N^q\left(\|(1-P_{N})a P_{N-N^\alpha}\|+\|
P_{N-N^\alpha} a (1-P_{N})\|\right),
\]
for all $a\in \csr\gh$.
\begin{definition} Given a Hecke pair $\gh$ with a length function $L$,
the vector space  $T\ill\gh$ associated with $L$
is defined as
\[
T\ill\gh:=\left\{a\in \csr\gh; \forall \alpha \in ]0,1[, \forall q \in \n, \sup_{N\geq 1} \rho_{\alpha, q, N}(a)
< \infty \right\}
\]
and is endowed with the locally convex topology induced by the norm
of $\csr\gh$ and seminorms
\[
\nu_{\alpha,q}(a):=\sup_{N\geq 1} \rho_{\alpha, q, N}(a),
\]
for  $0<\alpha<1$ and $q\in \mathbb{N}$.
\end{definition}

We note that the defining condition for an element $a\in\csr\gh$ to
be in $T\ill\gh$ is equivalent to the condition that $\|(1-P_{N})a
P_{N-N^\alpha}\|+\| P_{N-N^\alpha} a (1-P_{N})\|=O(N^{-q})$ when $N$
tends to infinity for all $\alpha\in (0,1)$ and all positive
integers $q$. We also note that if $f\in\ell^2(H\ba G)$, then
$\sum_{h\in <H\ba G>} (1-P_N)f(h)= \sum_{h\in <H\ba G>, L(h)>N} f(h)$ and
like wise for $P_{N-N^\alpha}$. In the following remark, we show
that $T\ill\gh$ is actually an algebra and discuss various features
of $T\ill\gh$.
\begin{remark}
\label{rem:tinf1}
\begin{itemize}
\item[(i)] If $\alpha > \beta$, then $P_{N-N^\beta} = P_{N-N^\alpha} + Q$,
where $Q$ is the orthogonal projection on the closed span of the set
$\{\delta_g; g\in <H\ba G>, N-N^\alpha < L(g) \leq N-N^\beta\}$ of
$\ell^2(H\ba G)$, and so $\nu_{\alpha,q}\leq \nu_{\beta,q}$. Therefore,
in order to generate the topology of $T\ill\gh$, one may choose the
parameter $\alpha$ from an arbitrary sequence in $]0,1[$ approaching
to 0. In this way, we get a countable family of seminorms defining
the topology of $T\ill\gh$. Now let $\{a_n\}$ be a Cauchy sequence in
$T\ill\gh$. Since it is Cauchy with respect to the norm of $\csr\gh$,
it has a limit, say $a$, in $\csr\gh$. On the other hand, since
$\{a_n\}$ is Cauchy with respect to every seminorm $\nu_{\alpha,q}$,
for given $\epsilon>0$ and $m, n$ large enough, we have
$\nu_{\alpha,q}(a_n-a_m)\leq \epsilon$. Now due to the facts that
$\|a_n-a\|\rightarrow 0$ and the seminorms are continuous \wrt the norm of $\csr\gh$, by letting $m$ tend to infinity, we get
$\nu_{\alpha,q}(a_n-a)\leq \epsilon$. This shows that
$a_n\rightarrow a$ with respect to every seminorm $\nu_{\alpha,q}$
and consequently $a\in T\ill\gh$. Hence, $T\ill\gh$ is a Fr\'echet
space.
\item[(ii)] Indeed, $T\ill\gh$ is a Fr\'echet algebra in the
terminology of \cite{sch}, namely in addition to the above remark,
the multiplication in $T\ill\gh$ is jointly continuous. For all $a,
b\in T\ill\gh$, and all $\alpha$ and $q$, we have
\begin{eqnarray*}
(1-P_N)abP_{N-N^\alpha}&=&(1-P_N)a P_{N-N^{\alpha/2}} b P_{N-N^\alpha}\\
&+&(1-P_N)a (1-P_{N-N^{\alpha/2}})b P_{N-N^\alpha}.
\end{eqnarray*}
Hence,
\begin{eqnarray*}
\|(1-P_N)abP_{N-N^\alpha}\|&\leq& \|(1-P_N)a P_{N-N^{\alpha/2}}\| \|b P_{N-N^\alpha}\|\\
&+& \|(1-P_N)a \|\|(1-P_{N-N^{\alpha/2}})b P_{N-N^\alpha}\|\\
&\leq& \|(1-P_N)a P_{N-N^{\alpha/2}}\| \|b\|\\
&+& \|a\|\|(1-P_{N-N^{\alpha/2}})b P_{N-N^{\alpha/2}-N^ {\alpha/2}}\|\\
&\leq& \|(1-P_N)a P_{N-N^{\alpha/2}}\| \|b\|\\
&+& \|a\|\|(1-P_{M})b P_{M-M^ {\alpha/2}}\|,
\end{eqnarray*}
where $M=N-N^{\alpha/2}$ and $M\rightarrow \infty$ whenever
$N\rightarrow \infty$. In the last step, we used the fact that
$(N-N^{\alpha/2})-N^ {\alpha/2}\leq
(N-N^{\alpha/2})-(N-N^{\alpha/2})^ {\alpha/2}$ which can be easily
checked. We have a similar estimation for $\|P_{N-N^\alpha} ab
(1-P_N)\|$. One also notes that $O(N^q)=O(M^q)$. These facts imply
that
\begin{equation}\label{f:submulti}
\nu_{\alpha,q}(ab) \leq
\nu_{\alpha/2,q}(a)\|b\|+\nu_{\alpha/2,q}(b)\|a\|.
\end{equation}
This shows that $T\ill\gh$ is closed under multiplication and the
multiplication is jointly continuous in the Fr\'echet topology of
$T\ill\gh$.
\item[(iii)] It is clear that $T\ill\gh$ is closed under involution,
addition and scalar multiplication, thus it is an involutive
subalgebra of $\csr\gh$. Now, let $g$ be an element of $G//H$ and
$\delta_g$ be the characteristic function of the double coset
containing $g$ as an element of $\H\gh$. Given $\alpha\in (0,1)$, It
is clear that $(1-P_N)\delta_g P_{N-N^\alpha}=0$ for all $N\geq
L(g)^{1/\alpha}$. Thus, $\delta_g$ belong to $T\ill\gh$. This proves
that $\H\gh \subseteq T\ill\gh$.
\item[(iv)] For every length function $L$ on $\gh$, we have $T\ill\gh\subseteq H\ill\gh$.
To see this, recall that every element $f$ of $\csr\gh$ can be
regarded as an element of $\ell^2(H\ba G)$ by its action on $\delta_1\in \ell^2(H\ba G)$,
i.e. $f(\delta_1)$. Thus, for given $\f\in T\ill\gh$, we can write
$\|(1-P_N)\f\|_2=\|(1-P_N)\f(\delta_1)\|_2=\|(1-P_N)\f(P_{N-N^{1/2}})(\delta_1)\|_2\leq
\|(1-P_N)\f(P_{N-N^{1/2}})\|$. Since, for every positive integer
$q$, the last term is equal to $O(N^{-q})$ when
$N\rightarrow\infty$, the following series is convergent for every
positive integer $m$:
\begin{eqnarray*}
\sum_{k\geq 1} k^{2m} \| (1-P_k)\f \|_2^2
&=\,& \sum_{k\geq 1} k^{2m} \sum_{h\in<H\ba G>, L(h)>k}|\f (h)|^2 \\
&=:&  \star
\end{eqnarray*}
Letting $\chi_l$ be the characteristic function of the set
$C_{l-1,L}(H\ba G)$, we obtain
\begin{eqnarray*}
\star &=&  \sum_{k\geq 1} k^{2m} \sum_{l>k}\|\f\chi_l \|_2^2 \\
&=&  \sum_{k\geq 2} \left(\sum_{l=1}^{k-1} l^{2m}\right)\|\f\chi_k
\|_2^2.
\end{eqnarray*}
Now, for given $s>0$, there exists an integer $m$ large enough such
that $\sum_{l=1}^{k-1} l^{2m}\geq (1+k)^{2s}$ and this shows
$\|\f\|_{s,L}\leq \infty$. Since this is true for all $s>0$, we conclude
that $\f\in H\ill\gh$.
\end{itemize}
\end{remark}

\begin{definition} For a Hecke pair $\gh$ equipped with a length
 function $L$, the Fr\'{e}chet algebra $T\ill\gh$ is called the
{\it Jolissaint algebra of the Hecke pair $\gh$ with respect to
$L$}.
\end{definition}

\begin{proposition}
\label{pro:HT}
If a Hecke pair $\gh$ has property (RD) with respect to a length
function $L$, then $T\ill\gh= H\ill\gh$.
\end{proposition}
\begin{proof}
In the above remark, we showed that $T\ill\gh\subseteq H\ill\gh \cap \csr\gh$. Now, by proving the converse of this inequality, we obtain the equality $T\ill\gh= H\ill\gh \cap \csr\gh$. From this, it is obvious that $T\ill\gh= H\ill\gh$ if and only if $H\ill\gh \subseteq \csr\gh$.

Let $\f$ be an element of $H\ill\gh$. For all $\alpha\in (0,1)$, natural number $N$ and $\xi\in\ell^2(H\ba G)$, we
have
\begin{eqnarray*}
\|(1-P_N)\f P_{N-N^\alpha}(\xi)\|_2^2&=\,& \sum_{g\in <H\ba G>} \left|(1-P_N)\f P_{N-N^{\alpha}}(\xi)(g)\right|^2 \\
&=\,&  \sum_{g\in <H\ba G>, L(g)>N} |\f P_{N-N^{\alpha}}(\xi)(g)|^2 \\
&=\,&  \sum_{L(g)>N} \left| \sum_{L(h)>N^\alpha }\f(h) P_{N-N^{\alpha}}(\xi)(h^{-1}g)\right|^2\\
&=:& \bigstar.
\end{eqnarray*}
The last line is because of the fact that if $L(h)\leq N^\alpha$, then
$N-N^\alpha < L(g)-L(h)\leq L(h^{-1}g)$, and so
$P_{N-N^{\alpha}}(\xi)(h^{-1}g)=0$. Now, because of the condition
$L(h)>N^\alpha $, for every positive integer $q$, we have
\[
\bigstar\leq \frac{1}{(1+N^\alpha)^{2q/\alpha}} \sum_{L(g)>N} \left| \sum_{L(h)>N^\alpha }\f(h) (1+L(h))^{q/\alpha} P_{N-N^{\alpha}}(\xi)(h^{-1}g)\right|^2.
\]
By defining $f(h):=\f(h) (1+L(h))^{q/\alpha}$ and going
backward, we obtain
\begin{eqnarray*}
\bigstar&\leq& \frac{1}{(1+N^\alpha)^{2q/\alpha}} \|(1-P_N)f
P_{N-N^\alpha}(\xi)\|_2^2\\
&\leq& \frac{1}{(1+N^\alpha)^{2q/\alpha}}
\|\lambda(f)\|^2\|\xi\|_2^2.
\end{eqnarray*}
Since $\f\in H\ill\gh$, we have $f\in H\ill\gh$ and, for some positive
real numbers $C, s$, we can write
\begin{eqnarray*}
\bigstar&\leq& \frac{C^2}{(1+N^\alpha)^{2q/\alpha}}
\|f\|_{s,L}^2\|\xi\|_2^2\\
&\leq& \frac{C^2}{(1+N^\alpha)^{2q/\alpha}} \|\f\|_{s+q/\alpha
,L}^2\|\xi\|_2^2.
\end{eqnarray*}
This shows that $\|(1-P_N)\f P_{N-N^\alpha}\|=O(N^{-q})$ when $N$
tends to infinity. A similar argument applies to
$\|P_{N-N^\alpha}\f(1-P_N)\|$. Therefore, $\f\in T\ill\gh$.
\end{proof}

\begin{lemma} Let $\mathcal{A}$ be a unital dense subalgebra of a Banach
algebra $A$. Let there exist a $0<\delta<1$ such that the series
$\sum_{n\geq 1} a^n$ belongs to $\mathcal{A}$ for all
$a\in\mathcal{A}$ that $\|a\|<\delta$. Then $\mathcal{A}$ has the
property of spectral permanence in $A$.
\end{lemma}
\begin{proof}
Let $x\in\mathcal{A}$ be invertible in $A$. Then there exists a
$y\in\mathcal{A}$ such that $\|x^{-1}-y\|<\delta/\|x\|$. This
implies that $\|1- xy\|<\delta$. By assumption, $\sum_{n\geq 0}
(1-xy)^n$ belongs to $\mathcal{A}$ and is a right inverse of $xy$.
This shows $x$ has a right inverse (and similarly a left inverse) in
$\mathcal{A}$. Thus, $x$ is invertible in $\mathcal{A}$.
\end{proof}
\begin{proposition}
\label{prop:spec} Let $a\in T\ill\gh$ and let $\|a\|\leq 1/2$. Then
$\sum_{n\geq 1} a^n\in T\ill\gh$ and consequently $T\ill\gh$ has the
property of spectral permanence in $\csr\gh$. Moreover, for all
positive integers $q$ and $\alpha\in (0,1)$, we have
\begin{equation}
\label{eq:nuineq2} \nu_{\alpha,q}\left( \sum_{n\geq 1} a^n \right)
\leq c(\alpha,q) \left(\nu_{\alpha/2, q+1} (a) + \|a\|\right).
\end{equation}
where $c(\alpha,q)>0$ is independent of $a$.
\end{proposition}
\begin{proof}
Let fix positive integers $q$, $N$ and real number $\alpha\in
(0,1)$. For $1\leq n\leq N^{\alpha/2}$, one can write
\begin{eqnarray*}
\|(1-P_N)a^n P_{N-N^{\alpha}}\| &=& \|(1-P_N)a P_{N-N^{\alpha/2}}
a^{n-1}P_{N-N^{\alpha}}\|\\
&+& \|(1-P_N)a (1-P_{N-N^{\alpha/2}})
a^{n-1}P_{N-N^{\alpha}}\|\\
&\leq& \|(1-P_N)a P_{N-N^{\alpha/2}}\|\\
&+& \|(1-P_{N-N^{\alpha/2}}) a^{n-1}P_{N-N^{\alpha}}\|.
\end{eqnarray*}
By repeating the same procedure for the last term until $a$ has no
power greater than 1, we get
\[
\|(1-P_N)a^n P_{N-N^{\alpha}}\| \leq \sum_{k=0}^{n-1} \|(1-P_{N_k})a
P_{N_{k+1}}\|,
\]
where $N_k=N-kN^{\alpha/2}$ for $0\leq k\leq n-1$. For $n>
N^{\alpha/2}$, clearly we have $\|(1-P_N)a^n P_{N-N^{\alpha}}\| \leq
2\|a\|/2^{N^{\alpha/2}}$. Combining these facts and noticing that
$N_{k+1}\leq N-N_k^{\alpha/2}$, we have
\[
\|(1-P_N)(\sum_{n\geq 1}a^n ) P_{N-N^\alpha}\| \leq \sum_{n\leq
N^{\alpha/2}} \sum_{k=0}^{n-1} \|(1-P_{N_k})a
P_{N-N_{k}^{\alpha/2}}\| +\frac{2\|a\|}{2^{N^{\alpha/2}}}.
\]
A similar estimation holds for the other term and the number of
terms in the above double sum is less than or equal $N$. We also
note that $O(N^{-q})=O(N_k^{-q})$ when $N$ tends to infinity for
every $k$. Therefore, up to a constant, which we call it
$c(\alpha,q)$, we obtain the desired inequality.
\end{proof}
\begin{theorem} $T\ill\gh$ is a smooth subalgebra of $\csr\gh$.
\end{theorem}
\begin{proof} Let $a\in T\ill\gh$ and let $f$ be a holomorphic
function defined on a neighborhood of $\sigma(a)$, the spectrum of
$a$. If we show that the map $\c \backslash \sigma(a)\rightarrow
T\ill\gh$ defined by $\lambda \mapsto (\lambda-a)^{-1}$ is continuous
in the Fr\'echet topology of $T\ill\gh$, then the Riemann sum of the
integral defining $f(a)$ converges in $T\ill\gh$, and so $f(a)$
belongs to $T\ill\gh$.

Fix $\sigma_0\in \c\backslash \sigma(a)$. For every $\sigma \in
\c\backslash \sigma(a)$, set
$x(\sigma):=(\sigma_0-\sigma)(\sigma_0-a)^{-1}$. If
$|\sigma_0-\sigma|\leq \frac{1}{2}\|\sigma_0-a\|$, then
$\|x(\sigma)\|\leq 1/2$ and by Proposition \ref{prop:spec}
$\sum_{n\geq 1} x(\sigma)^n\in T\ill\gh$. Moreover, for every
$\alpha\in(0,1)$ and positive integer $q$, we have
\begin{eqnarray*}
\nu_{\alpha, q} \left((\sigma_0-a)^{-1}- (\sigma-a)^{-1})\right) &=&
\nu_{\alpha, q} \left( (\sigma_0-a)^{-1}[1- (1-x(\sigma))^{-1}]
\right)\\
&=& \nu_{\alpha, q} \left( (\sigma_0-a)^{-1}\sum_{n\geq 1} x(\sigma)^n \right)\\
\text{ by \ref{f:submulti}}&\leq& \nu_{\alpha/2, q} \left(
(\sigma_0-a)^{-1}\right) \|\sum_{n\geq 1} x(\sigma)^n\| \\
&+& \|(\sigma_0-a)^{-1}\| \nu_{\alpha/2, q} \left( \sum_{n\geq 1}
x(\sigma)^n \right)\\
\text{ by \ref{eq:nuineq2}}&\leq& \nu_{\alpha/2, q} \left(
(\sigma_0-a)^{-1}\right) \|\sum_{n\geq 1} x(\sigma)^n\| \\
&+& \|(\sigma_0-a)^{-1}\| c(\alpha/2,q) \left[\nu_{\alpha/4, q+1}
\left( x(\sigma) \right) + \|x(\sigma)\| \right].
\end{eqnarray*}
Because of $x(\sigma)$, each term in the right hand side of the
above inequality has a factor $|\sigma_0-\sigma|$ and this completes
the proof.
\end{proof}
\begin{corollary} $K_0(T\ill\gh)\simeq K_0(\csr\gh)$.
\end{corollary}
\begin{proof} Due to the facts that  $T\ill\gh$ is a Fr\'echet
algebra, Remark \ref{rem:tinf1} Parts (i) and (ii), and it is a
smooth subalgebra of $\csr\gh$, we can use the result of Larry B.
Schweitzer in \cite{sch} and deduce that $M_n(T\ill\gh)$ is smooth in
$M_n(\csr\gh)$ for all positive integers $n$. Now, the statement
follows from \cite{bl}.
\end{proof}
The above argument can be considered as another proof for the part 4
of the proof of Theorem 1.4 in \cite{j1}.
\begin{corollary}
\label{cor:Kgroups}
If $\gh$ has property (RD) with respect to a length
function $L$, then  the algebra $H\ill\gh$ of rapidly decreasing
functions on $G//H$ with respect to $L$ is a smooth subalgebra of
$\csr\gh$ and the inclusion $H\ill\gh \subset \csr\gh$ gives rises to
the isomorphism $K_0(H\ill\gh)\simeq K_0(\csr\gh)$.
\end{corollary}

%%%%%%%%%%%%%%%%%%%%%%%%%%%%%%%%%%%%%      Other methods       %%%%%%%%%%%%%%%%%%%%%%%%%%%%%%%%%%%%%%%%

\section*{Other methods}
\label{sec:other}

Since Jolissaint's work in \cite{j1}, some new methods have been invented to prove that  $H^\infty_L(G)$ is a smooth subalgebra of $\csr(G)$ if $G$ has property (RD). In this section, we discus this methods.

The first method uses unbounded derivations to define a smooth subalgebra of  $\csr(G)$. It has been used in \cite{cm} as well as Alain Connes' book and some recent articles. Here, we follow the paper \cite{ji} by Ronghui Ji which contains more details. Using a length functions $L$ on $G$, one defines an unbounded operator $d_L:\ell^2(G)\ra \ell^2(G)$ by $d_L(f)(g):= L(g)f(g)$ for all $f\in \ell^2(G)$. It is a closed self-adjoint unbounded operator. Therefore the map $\delta_L :B(\ell^2(G))\ra B(\ell^2(G))$ defined by
\[
\delta_L(T):= i[d_L, T], \quad \forall T\in B(\ell^2(G))
\]
is a closed unbounded $\ast$-derivation. Let us denote the domain of an arbitrary unbounded operator (derivation) $d$ by $D(d)$. Then it is easy to see that $\c G$ is contained in $\bigcap_{k=1}^\infty D(\delta_L^k)$. Set $S_L(G):= \bigcap_{k=1}^\infty D(\delta_L^k) \cap \csr(G)$. Since $S_L(G)$ is a dense $\ast$-subalgebra of $\csr(G)$, the following theorem is applied to show that $S_L(G)$ is actually a smooth subalgebra of $\csr(G)$.

\begin{theorem} (Theorem 1.2 of \cite{ji})
\label{thm:dersmoothsub}
Let $B$ be a \cs-algebra and $A$ a \cs-subalgebra of $B$. Let $\delta :B\ra B$ be a closed unbounded derivation. Then $\bigcap_{k=1}^\infty D(\delta^k) \cap A$ is a subalgebra of $A$. Moreover, it is a smooth subalgebra if it is dense.
\end{theorem}

Regarding this theorem, it is enough to show that when $G$ has property (RD), $S_L(G)$ and $H^\infty_L(G)$ are actually the same Fr\'echet algebras. This was done in Theorem 1.3 of \cite{ji}.

In order to extend the above argument to the frame work of Hecke pairs, one has to replace $\ell^2(G)$ by $\ell^2(H\ba G)$ and define an unbounded operator $d_L$ and an unbounded derivation $\delta_L$ similarly. The rest is just an adaptation of Ronghui Ji's argument for Hecke pairs.

The second method is due to Vincent Lafforgue. In Proposition 1.2 of \cite{l2}, It is shown directly that if $G$ has property (RD) with respect to a length function $L$, then there is a positive real number $s$ such that $H^s_L(G)$ is a Banach algebra and is a smooth subalgebra of $\csr(G)$. Again Lafforgue's argument can be generalized for Hecke pairs too.

%%%%%%%%%%%%%%%%%%%%%%%%%%%%%%%%%%%%%      Acknowledgements       %%%%%%%%%%%%%%%%%%%%%%%%%%%%%%%%%%%%%%%

\noindent {\bf Acknowledgements.} I would like to thank both the referees of this paper for their helpful comments and suggestions. Their comments have enriched this paper significantly and stimulated some ideas which will be followed in my future papers in this subject.

%%%%%%%%%%%%%%%%%%%%%%%%%%%%%%%%%%%%%%%%%%%%%%%%%%%%%%%%%%%%%%%%%%%%%%%%%%%%%%%%%%%%%%%%%%%%%%%%%%%%%%%%%%%%%%%%%
%%%%%%%%%%%%%%%%%%%%%%%%%%%%%%%%         bibliography           %%%%%%%%%%%%%%%%%%%%%%%%%%%%%%%%%%%%%%%%%%%%%%%%%
\bibliographystyle {amsalpha}

\begin {thebibliography} {VDN92}

\bibitem{ach} {\bf C. Antonescu, E. Christensen,} Metrics on group \cs-algebras and a non-commutative Arzel\`{a}-Ascoli theorem, J. of Functional Analysis, 214 (2004) 247--259.

\bibitem{ar} {\bf W. Arveson,} A short course on spectral theory. Graduate texts in mathematics, vol. 209, Springer-Verlag New York, Berlin, Heidelberg, (2001).

\bibitem{bec1} {\bf J. J. Becnel,} Countably normed spaces, their dual and Gaussian measures. (Preprint ArXiv: math.0407200)

\bibitem{bl} {\bf B. Blackadar,} $K$-theory for operator algebras. Mathematical Sciences Research Institute Publications, vol. {\bf 5} Cambridge University press, second edition, (1998).

\bibitem{blcu} {\bf B. Blackadar, J. Cuntz,} Differential Banach algebra norms and smooth subalgebras of C*-algebras, J. Operator Theory {\bf 26} (1991), 255--282.

\bibitem{b}{\bf J.-B. Bost,} Principe d'Oka, K-th\'{e}orie et syst\`{e}mes dynamiques non commutatifs. Invent. Math., 101 (2), (1990), 261--333.

\bibitem{bc}{\bf J.-B. Bost, A. Connes,} Hecke algebras, type III factors and phase transition with spontaneous symmetry breaking in number theory. Selecta Math. (New Series) Vol.{\bf 1}, no.3, 411-456, (1995).

\bibitem{bump} {\bf D. Bump,} Automorphic forms and representations. Cambridge Studies in Advanced Mathematics: {\bf 55} Cambridge University Press (1998).

\bibitem{chat} {\bf I. L. Chatterji,} On property (RD) for certain discrete groups. PhD Thesis, Swiss Federal Institute of Technology Zurich, (2001).

\bibitem{chatru} {\bf I. L. Chatterji, K. Ruane,} Some geometric groups with rapid decay. Geom. funct. anal. vol. {\bf 15}, (2005), 311--339.

\bibitem{c85} {\bf A. Connes,} Non-commutative differential geometry. Inst. Hautes \'{E}tudes Sci. Publ. Math. No. 62, (1985), 257--360.

\bibitem{c94} {\bf A. Connes,} Noncommutative geometry. Academic Press, Inc. (1994).

\bibitem{cm} {\bf A. Connes, H. Moscovici,} Cyclic cohomology, the Novikov conjecture and hyperbolic groups.  Topology  {\bf 29} (1990), no. 3, 345--388.

\bibitem{fo} {\bf G. B. Folland,} Real analysis, Modern techniques and their applications. second edition, John  Wiley \& Sons, Inc. (1999).

\bibitem{haag1} {\bf U. Haagerup,} An example of a nonnuclear \cs-algebra, which has the metric approximation property. Invent. Math. 50 (1979), 279--293.

\bibitem{Hall} {\bf R. W. Hall,} Hecke \cs-algebras. PhD thesis, The Pennsylvania State University, (1999).

\bibitem{harpe} {\bf P. de la Harpe,} Groupes hyperboliques, alg\`{e}bres d'op\'{e}rateurs et un th\'{e}or\`{e}me de Jolissaint. C. R. Acad. Sci. Paris S\'{e}r. I, vol. {\bf 307}, (1988), 771--774.

\bibitem{ji} {\bf R. Ji,} Smooth dense subalgebras of reduced group \cs-algebras, Schwartz cohomology of groups, and cylic cohomology. Journal of Functional analysis {\bf 107}, (1992) 1--33.

\bibitem{j1} {\bf P. Jolissaint,} $K$-theory of reduced \cs-algebras and rapidly decreasing functions on groups. Journal of $K$-Theory {\bf 2}, (1989) 723--735.

\bibitem{j2} {\bf P. Jolissaint,} Rapidly decreasing functions in reduced \cs-algebras of groups. Trans. Amer. Math. Soc. Vol. {\bf 317}, (1990), no. 1, 167--196.

\bibitem{l2} {\bf V. Lafforgue,} A proof of property (RD) for discrete cocompact subgroups of $SL_3(\r)$. Journal of Lie Theory, vol. {\bf 10}, (2000), 255--267.

\bibitem{l1} {\bf V. Lafforgue,}  $K$-th\'{e}orie bivariante pour les alg\`{e}bres de Banach et conjecture de Baum-Connes.  Invent. Math.  {\bf 149},  (2002),  no. 1, 1--95.

\bibitem{sch} {\bf L. B. Schweitzer,} A short proof that $M\sb n(A)$ is local if $A$ is local and Fr\'echet.  Internat. J. Math. {\bf 3}  (1992),  no. 4, 581--589.

\bibitem{s1} {\bf V. Shirbisheh,} Various commensurability relations in Hecke pairs and property (RD). arXiv:1205.2819, (2012).

\bibitem{tzanevthesis} {\bf K. Tzanev,} \cs-alg\`{e}bres de Hecke et $K$-th\'{e}ories. Th\`{e}se de doctorat, Universit\'{e} Paris 7, (2000).

\bibitem{tzanev} {\bf K. Tzanev,} Hecke \cs-algebras and amenability,  J. Operator Theory, {\bf 50}, (2003), 169--178.

\bibitem{valette} {\bf A. Valette,} On the Haagerup inequality and groups acting on $\widetilde{A}_n-$buildings. Ann. Inst. Fourier, Grenoble {\bf 47}, (1997), 4, 1195--1208.

\bibitem{ve} {\bf R. Vergnioux,} The property of rapid decay for discrete quantum groups. J. Operator Theory, {\bf 57}, (2007), no. 2, 303--324.
%%%%%%%%%%%%%%%%%%%%%555555555555555555555555555555555555555555555555555555555555555555555555555555555555555555555555555555555555555555555
\end {thebibliography}

\end{document}